\documentclass[12pt]{article}

\usepackage{latexsym,amssymb}
\usepackage[T1]{fontenc}
\usepackage[dvips]{graphicx}
\newcommand{\z}{\mathbb Z}
\newcommand{\q}{\mathbb Q}
\newcommand{\n}{\mathbb N}
\newtheorem{lem}{Lemma}[section]

\newtheorem{ex}[lem]{Example}
\newtheorem{co}[lem]{Corollary}
\newtheorem{thm}[lem]{Theorem}

\newenvironment{proof}{\textbf{Proof.}}{\newline\hspace*{\fill}{$\Box$}}

\begin{document}
\title{Largeness of LERF and 1-relator groups}
\author{J.\,O.\,Button\\
Selwyn College\\
University of Cambridge\\
Cambridge CB3 9DQ\\
U.K.\\
\texttt{jb128@dpmms.cam.ac.uk}}
\date{}
\maketitle
%\newpage
\begin{abstract}
We consider largeness of groups given by a presentation of deficiency 1,
where the group is respectively free-by-cyclic, LERF or 1-relator. We
give the first examples of (finitely generated free)-by-$\z$ word
hyperbolic groups which are large, show that a LERF deficiency 1 group
with first Betti number at least 2 is large or $\z\times\z$ and show that
2-generator 1-relator groups where the relator has height 1 obey the
dichotomy that either the group is large or all its finite images are
metacyclic.
\end{abstract}
%\newpage
\section{Introduction}

A finitely generated group $G$ is said to be large 
if it has a finite index
subgroup possessing a homomorphism onto a non-abelian free group. 
This has a range of implications: for instance
$G$ is SQ-universal (which means that every
countable group is a subgroup of a quotient of $G$), 
$G$ has uniformly exponential word growth, $G$ has
the largest possible subgroup growth for finitely generated groups
(which is of strict type $n^n$)  
and $G$ has infinite virtual first Betti number. If we restrict
ourselves to finitely presented groups and define the deficiency of a
finite presentation to be the number of generators minus the number of
relators, a well known result of
B.\,Baumslag and S.\,J.\,Pride in \cite{bp} is that groups with a 
presentation of deficiency at least 2 are large.

The results in this paper follow on from \cite{me} where the question
considered was which groups with a deficiency 1 presentation are large.
Clearly $\z$ and $\z\times\z$ are not, and neither are the soluble
Baumslag-Solitar groups $BS(1,n)$, where the Baumslag-Solitar group
$BS(m,n)$ has the presentation $\langle x,y|yx^my^{-1}=x^n\rangle$.
There are other cases, such as $BS(2,3)$, but no more residually finite
examples are known (indeed all others are ``far from being residually
finite'' in a sense that will be made precise in Section 4).

However in \cite{me} many families of deficiency 1 groups which are all
large were found. In particular if $F_n$ is the free group of rank $n$
then it was shown that free-by-cyclic groups $F_n\rtimes_\alpha\z$ are
large for $n\geq 2$ if they contain $\z\times\z$, which is equivalent
here to not being a word hyperbolic group. The question of whether
$F_n\rtimes_\alpha\z$ is large in the word hyperbolic case was left
open, and up until now not a single large example was known. In Section 2
we show that if $\alpha$ is a reducible automorphism then $F_n\rtimes
_\alpha\z$ is large provided that free-by-cyclic groups have finite index
subgroups with first Betti number at least 2. This was raised by A.\,Casson
in \cite{bspl} Question 12.16. Although still unknown in the word hyperbolic
case, if $\alpha$ is a specific reducible automorphism then we merely require
for largeness that two particular free-by-cyclic groups have virtual first
Betti number at least 2: these are the one obtained by restricting (a
suitable power of) $\alpha$ to the invariant free factor, and the 
free-by-cyclic group formed by quotienting out the invariant free factor.
By taking a specific word hyperbolic group in the literature which is
of the form $F_3\rtimes_\alpha\z$ and using it to make a reducible
automorphism, we obtain Corollary 2.4 which gives the first
group of the form $F_n\rtimes_\alpha\z$ which is known to be
both word hyperbolic and
large. It double covers a group with the same properties which has a
most succinct presentation: $\langle t,a|t^6at^{-4}
a^{-1}t^{-2}a^{-1}\rangle$.

We have mentioned that the property of residual finiteness should increase
the chances of a deficiency 1 presentation being large. For instance
all groups of the form $F_n\rtimes_\alpha\z$ are residually finite. In
Section 3 we look at deficiency 1 groups which are LERF (also known as
subgroup separable). This is considerably stronger than residual
finiteness so we would expect these groups to be large (with three
obvious exceptions). Once again though the problem is finding a finite
cover with first Betti number at least 2. We prove in Theorem 3.3 that
if $G$ is a LERF group with a presentation of deficiency 1 and has such
a finite cover then $G$ is large or the fundamental group of the torus
or Klein bottle. A recent result in \cite{ko}
of D.\,Kochloukova is used in
Theorem 3.1 to show that the only possible exceptions to LERF deficiency 1
groups being large (apart from $\z$ and these two groups) are word hyperbolic
groups of the form $F_n\rtimes_\alpha\z$. This at least gives us in
Corollary 3.2 that all but these three deficiency 1 LERF groups are
SQ-universal, but we are prevented from concluding largeness until
Casson's question is settled (at least in the LERF case but even here
this seems open). We also look at the question of a type of
Tits alternative for deficiency 1 groups which would say that such a group
is either soluble or contains a non-abelian free group. The result in
\cite{ko} 
mentioned above very nearly established this but one case is still
to be resolved. We show in Corollary 3.4 that this is true if $G$ has a
finite index subgroup with first Betti number at least 2. However it is
not true that all finitely generated subgroups of $G$ will either be virtually
soluble or contain a non-abelian free group, and Example 3.5 is such a
group which has deficiency exactly 1.

Another much studied class of groups are those with a 1-relator
presentation. The intersection of 1-relator and deficiency 1 groups is
the class of 2-generator 1-relator presentations. Although groups of this
form have strong properties, it is not known which ones are large.
Moreover we would like to be able to deduce largeness using only
information obtained directly from the presentation rather than
needing to know a priori that the group has special properties such as
residual finiteness.

In Section 4 we settle this question for a particular class of 2-generator
1-relator presentations. Given any relator in 2 variables we can make a
change of basis of $F_2$ to $\{a,t\}$ such that the exponent sum of $t$
in the relator is zero. We say that $r$ has height 1 if appearances of
$t^{\pm 1}$ in $r$ are such that $t$ alternates with $t^{-1}$. Although
this is a restricted set of relators, it is the case that nearly all of
the 2-generator 1-relator groups in the literature which have unusual
or nasty properties (we review these in that section) are given by
height 1 words.

We establish a major dichotomy of groups $G$ with height 1 presentations in
Theorem 4.1 which states that either $G$ is large or all its finite images
are metabelian. If $G$ is not a soluble Baumslag-Solitar group but is in
the latter case then $G$ contains a non-abelian free group, so is far
from being metabelian and hence far from being residually finite. We also
apply a famous result of Zelmanov on pro-$p$ groups which allows us in
Corollary 4.2 to distinguish between the two cases: $G$ is large if and
only if it has a finite index subgroup whose abelianisation requires
at least 3 generators, which is a condition that can easily be checked
on a computer.

It is not true that a large height 1 group is necessary residually finite.
However in Corollary 4.6 we give the first example of a 2-generator
1-relator presentation of a group $G$ where $\beta_1(G)=2$ and the only
finite images of $G$ are abelian but $G$ is not equal to $\z\times\z$. 
This means that it is not enough to conclude largeness (or equality
with $\z\times\z$) for a deficiency 1 group with first Betti number at
least 2 (even in the 1-relator case). However this is unknown if the first
Betti number is at least 3, or even if the abelianisation of the group
requires at least 3 generators.

Section 5 is a collection of open questions encountered during the
preparation of this work which, although some of these might be well
known, we have not found in the standard problem lists. They are all on
finitely generated and finitely presented groups and the list begins with
the most wide ranging questions and then gradually specialises, ending
with unsolved problems that are the most relevant to this paper.

The author would like to thank J.\,Hillman for bringing his attention
to the paper \cite{ko}.

\section{Reducible Free-by-Cyclic Groups}

A powerful method for proving largeness directly from a given finite
presentation is to apply Howie's result which is Theorem A
in \cite{how}. This tells us
that if there is a homomorphism $\chi$ from a finitely presented group
$G$ onto $\z$ such that the Alexander polynomial 
$\Delta_{G,\chi}(t)\in\z[t^{\pm 1}]$ relative to $\chi$
is identically zero then $G$ is
large (indeed the proof shows that a finite index subgroup containing
ker $\chi$ surjects onto a non-abelian free group). However not only do
we need the abelianisation $\overline{G}=G/G'$ of $G$ to be infinite
in order to have a homomorphism onto $\z$ in the first place, we also have
that if $\overline{G}=\z\times T$ for $T$ finite then $|\Delta_{G,\chi}
(1)|$ is the order of $T$. Thus we will not be able to use this criterion
for largeness unless the first Betti number $\beta_1(G)$ is at least two,
or we can find a finite index subgroup with this property. (In fact
Howie's theorem is also true if the mod $p$ Alexander polynomial
$\Delta_{G,\chi}^p\in\mathbb F[t^{\pm 1}]$ for $\mathbb F=\z/p\z$ is
zero and we will use this in later sections, but for now we will stick to
the characteristic zero version.)

Given a finitely presented group $G=\langle x_1,\ldots ,x_m|r_1,\ldots,r_n
\rangle$ which does have a homomorphism $\chi$ onto $\z$ where
$K=\mbox{ker }\chi$, we can regard $K/K'$ as a $\z[t^{\pm 1}]$-module
where $t$ acts by conjugation on $K$ using an element of $\chi^{-1}(1)$.
Moreover we can obtain a finite presentation for this module $K/K'$ by
using the Reidemeister-Schreier rewriting process to go from a group
presentation for $G$ to a group presentation for $K$, and then abelianising
the relations. Although this will result in infinitely many group
relations there are only finitely many orbits under the action of $t$.
The result is an $(m-1)$ by $n$  presentation matrix $M$ for $K/K'$ and we
can assume that $n\geq m-1$ by adding zero columns if necessary. We then
define the Alexander polynomial $\Delta_{G,\chi}\in\z[t^{\pm 1}]$ to be the 
highest common factor of the $(m-1)$ by $(m-1)$ minors of $M$; it is the
same (up to units in $\z[t^{\pm 1}]$) for any finite presentation of $G$.

The fact that a zero Alexander polynomial implies largeness is particularly
useful for groups $G$ with a presentation of deficiency 1 because the
resulting matrix is square and so we are merely evaluating the determinant
to obtain the Alexander polynomial. In particular if we find that one row
or column consists entirely of zeros then we immediately conclude
largeness. We do have the problem mentioned above that we need $\beta_1(G)
\geq 2$ for this to happen, however another advantage of deficiency 1
presentations is that for any finite index subgroup $H$ of $G$ (for which
we write $H\leq_f G$) the Reidemeister-Schreier rewriting process
results in a deficiency 1 presentation for $H$. As $H$ is large if and
only if $G$ is, we can hope that there is a subgroup $H$ with $\beta_1(H)
\geq 2$.

A large class of deficiency 1 presentations come from the free-by-cyclic 
groups: let $F$ be a free group and $\alpha$ an automorphism of
$F$. Then we can form the semidirect product (also called the mapping
torus) $F\rtimes_\alpha\z$. If $F_n$ is the free group of rank $n$ with free
basis $x_1,\ldots ,x_n$ then $F_n\rtimes_\alpha\z$ has the presentation        
\begin{equation}
\langle x_1,\ldots ,x_n,t|tx_1t^{-1}=\alpha(x_1),\ldots ,
tx_nt^{-1}=\alpha(x_n)\rangle.
\end{equation}
The following facts are known about free-by-cyclic groups 
$F\rtimes_\alpha\z$; see \cite{me}
Section 5 and references within.\\
(1) If $F$ is of infinite rank then $F\rtimes_\alpha\z$ may be finitely
or infinitely generated.
If $F\rtimes_\alpha\z$ is infinitely generated then it need not be
large, nor residually finite, but if it is finitely generated then it is
residually finite and finitely presented. Moreover it has a presentation
with deficiency at least 2, so is large.\\
(2) $F_n\rtimes_\alpha\z$ is residually finite, has deficiency exactly equal to
1 and any finite index subgroup is also of the form $F_m\rtimes_\beta\z$
so also has deficiency exactly 1.\\
(3) $F_n\rtimes_\alpha\z$ is word hyperbolic precisely when it does not
contain a subgroup isomorphic to $\z\times\z$. If it does and $n\geq 2$
then it is large. However it is not known whether word hyperbolic groups
of the form $F_n\rtimes_\alpha\z$ are large. 
A problem in \cite{bspl} due to Casson
is whether a group $G=F_n\rtimes_\alpha\z$ with $n\geq 2$ always has
$H\leq_f G$ with $\beta_1(H)\geq 2$. Whilst this is true if $G$ contains
$\z\times\z$, it is unknown in general if $G$ is word hyperbolic and
so we will need to take this as an assumption.

An automorphism $\alpha$ of the free group $F_n$ for $n\geq 2$ is said
to be reducible if there exist proper non-trivial free factors $R_1,
\ldots ,R_k$ of $F_n$ such that the conjugacy classes of $R_1,\ldots
,R_k$ are permuted transitively by $\alpha$ (see \cite{bh}).

\begin{thm} Assume that any group of the form $F_n\rtimes_\alpha\z$ for
$n\geq 2$ has a finite index subgroup with first Betti number at least
2. Then if $G=F_n\rtimes_\alpha\z$ for $\alpha$ a reducible automorphism,
we have that $G$ is large.
\end{thm}
\begin{proof}
Our group $G$ will have a presentation as in (1) and so for each positive
integer $k$ there exists the cyclic cover of $G$ which is the index $k$
subgroup $F_n\rtimes_{\alpha^k}\z$ generated by $t^k$ and $F_n$. If
$\alpha$ is reducible then on replacing $\alpha^k$ by $\alpha$ we can
assume $F_n=A*B$ with $A$ and $B$ proper free factors and $\alpha(A)$ is
sent to a conjugate of $A$. However on now composing $\alpha$ with an
appropriate inner automorphism (which does not change the free-by-cyclic
group) we can assume that $\alpha(A)=A$. From now on this will be our
$G$.

We define the free-by-cyclic group $G_r$ by restricting
$\alpha$ to $A$. Our assumption means that we can take a finite index
subgroup of $G_r$ with first Betti number at least 2 (and note that if
$A$ has rank 1 then we can do this as well). But every finite index
subgroup of $F_n\rtimes_\alpha\z$ contains one of the form
$F_m\rtimes_{\alpha^k}\z$ where $F_m\leq_f F_n$ and $\alpha^k(F_m)=F_m$.
Moreover the first Betti number does not decrease in finite covers.
Doing this for $G_r$, we have $A_0\leq_fA$ and a power of $\alpha$
fixing $A_0$
(which we again replace by $\alpha$) to get a finite index
subgroup $L=\langle t,A_0\rangle$ of $G_r$ with $\beta_1(L)\geq 2$.
However a result of Marshall Hall Jnr. states that as $A_0$ is a finitely
generated subgroup of $F_n=A*B$, there is a finite index subgroup $E$
of $F_n$ with $A_0$ a free factor, so we have $E=A_0*S$ for some $S$ and,
as $A_0$ does not have finite index in $F_n$, $S$ is non trivial. Again
taking a power of $t$, we can assume that $\alpha^j(E)=E$ as there are
only finitely many subgroups in $F_n$ of each finite index.

We are now ready to look at a finite presentation for the finite index
subgroup $J=\langle t^j,E\rangle$ of $G$. On taking a free basis $a_1,\ldots
,a_l$ for $A_0$ and a similar one for $S$, we find that the first $l$
relations in the deficiency 1 presentation for $L$ as in (1) are of the form
$ta_it^{-1}=w_i(a_1,\ldots ,a_l)$ where $w_i$ are reduced words in $F_l$.
However $\beta_1(L)$ being at least 2 means that we can change the free
basis for $A_0$ so that the first relation is actually
$ta_1t^{-1}=a_1c(a_1,\ldots ,a_l)$ where $c$ is a reduced word in the
commutator subgroup $F_l'$ of $F_l$. Now suppose there exists a
surjective homomorphism $\chi$
from $J$ to $\z$ with the property that $t$ and
all of $A_0$ are in ker $\chi$. When we form the square matrix $M$ with
entries in $\z[x^{\pm 1}]$ in order to calculate the polynomial
$\Delta_{J,\chi}(x)$, all the letters appearing in the first relation
are in the kernel of $\chi$. Therefore we have that the first column
of $M$ consists in turn
of the exponent sum of $t,a_1,\ldots ,a_l$ and 
then zeros in the
other places because the latter rows correspond to generators in
this presentation for $J$ which do not even appear in the first relation.
However the exponent sums of the generators $t,a_1,\ldots ,a_l$ that do
appear are all zero so we have a zero column in a square matrix,
meaning that $\Delta_{J,\chi}$ is the zero polynomial so $J$, and $G$,
are large.

In order to find such a homomorphism, we must again use our assumption
on finite covers with first Betti number at least 2. From $G$ we obtained
the ``reduced'' free-by-cyclic group $G_r$ by restricting $\alpha$ to $A$.
Although the definition of a reducible automorphism means that we cannot
assume $\alpha(B)=B$ even if $\alpha(A)=A$, we can form the ``quotient''
free-by-cyclic group $G_q$ by taking the ``quotient'' automorphism
$\tilde{\alpha}$ of $B$. This is formed by letting $\pi:A*B\rightarrow B$
be the homomorphism with kernel the normal closure of $A$ and then we
define $\tilde{\alpha}(b)=\pi\alpha(b)$ for $b\in B$. Note that 
$\tilde{\alpha}$ is surjective because $\alpha$ and $\pi$ are, and the
Hopfian property of finitely generated free groups means that
$\tilde{\alpha}$ is an automorphism. Also there is a natural homomorphism
$\theta$ from $G$ to $G_q=\langle s,B\rangle$ given by sending $t$ to $s$
and ``ignoring'' $A$; this is well defined because $\alpha(A)=A$.

Now we apply our assumption to obtain a finite index subgroup $H$ of $G_q$
with $\beta_1(H)\geq 2$ and without loss of generality we can assume $H$
is of the form $\langle s^i,C\rangle$ for $C\leq_f B$ and some $i\in\n$.
In fact as cyclic covers of $H$ will also have first Betti number at least
2, we can replace $i$ with $j$ by taking multiples so that they are both
equal to $ij$. As $\beta_1(H)\geq 2$, we must have a homomorphism
$\tilde{\chi}$ from $H$ onto $\z$ with $\tilde{\chi}(s^j)=0$. We then
consider the subgroup of $G$ which is $\theta^{-1}(H)\cap J$. As each of
these subgroups has finite index in $G$, so does their intersection. We
can take our finite presentation above for $J$ and use this to rewrite
for the finite index subgroup $\theta^{-1}(H)\cap J$. 
Note that
$t^j$ is in $\theta^{-1}(H)$ and this presentation will also have
deficiency 1. Also
the generators that we used for $A_0$
are in ker $\theta$ so are in $\theta^{-1}(H)$ too. Hence in going from
our old presentation to our new one, our valuable relation
survives intact. But the homomorphism $\tilde{\chi}\theta:\theta^{-1}(H)
\rightarrow\z$ can be restricted to $\theta^{-1}(H)\cap J$ which has
finite index in $\theta^{-1}(H)$. Without loss of generality
this restriction is surjective and moreover all the generators in our
special relation end up being sent to the identity by $\tilde{\chi}
\theta$ so we are done.
\end{proof}

\begin{co}
Assume that any group $F_n\rtimes_\alpha\z$ for $n\geq 2$ has a finite
index subgroup with first Betti number at least two. Suppose that
the group $G$ can be written in the form $F_n\rtimes_\alpha\z$ and there
exists a finitely generated subgroup $A$ of $F_n$ which is non-trivial
and of infinite index in $F_n$, whose conjugacy class has a finite
orbit under $\alpha$. Then $G$ is large.
\end{co}
\begin{proof}
By taking cyclic covers and an inner automorphism, we can assume that
$\alpha(A)=A$. Although $A$ may not be a free factor of $F_n$, we can
again use M.\,Hall Jnr.'s result  to find $S$ (with
$A$ and $S$ non-trivial) such that $A*S=F_m$ and $F_m$ has finite index
in $F_n$. We now take the appropriate power $\alpha^i$ of $\alpha$ that
fixes $F_m$. Then the subgroup $H=F_m\rtimes_{\alpha^i}\z$ has finite
index in $G$, with $\alpha^i$ reducible when restricted to $F_m$, so
$H$ is large by Theorem 2.1.
\end{proof} 

These results are all very well but we have not seen a single example
of a word hyperbolic group of the form $F_n\rtimes_{\alpha}\z$ which is
large. Although it seems as if we need to wait for Casson's question on
finite index subgroups with first Betti number at least two to be settled
positively, we can manage without this in specific cases. If we have an
automorphism $\alpha$ of $F_n=A*B$ with $\alpha(A)=A$, where $A$ and $B$
are proper free factors, and we form the free-by-cyclic group $G$ then the
proof of Theorem 2.1 was set out so that it is enough to find finite
index subgroups with first Betti number 2 of both the restriction
$G_r=A\rtimes_\alpha\z$ and the quotient $G_q=B\rtimes_{\tilde{\alpha}}\z$.
If we have a presentation for $G$ as in (1) and our free basis for $F_n$
is obtained by putting together ones for $A$ and $B$ then we instantly
get finite presentations for $G_r$ and $G_q$ which we can then feed into
a computer and ask it to enumerate finite index subgroups and their
abelianisations.

This allows us to check largeness of free-by-cyclic groups $G$ formed
by reducible automorphisms $\alpha$. If we require a word hyperbolic
example then it is necessary that $G$ contains no $\z\times\z$ subgroup,
or equivalently $\alpha$ has no periodic conjugacy classes. Moreover this
is sufficient for a group of the form $F_n\rtimes\z$
to be word hyperbolic by \cite{bf}, \cite{bfad} and \cite{brink}.
Thus it is straightforward
to create reducible word hyperbolic examples by a ``doubling'' process.
\begin{lem}
If $A\rtimes_\alpha\z$ is word hyperbolic 
where $A$ is isomorphic to $F_n$ and $\alpha$ is 
an automorphism of $A$ then $(A*B)\rtimes_\alpha\z$ is also
word hyperbolic, where $B$ is a copy of $A$ and the action of $\alpha$ on
$B$ is the same as on $A$.
\end{lem}
\begin{proof}
Suppose we have $w\in A*B$ with $\alpha^k(w)$ equal to a conjugate of $w$
for $k\geq 1$. Then $w$ is not in $A$ or $B$ so we can ensure (by conjugation
in $A*B$ if necessary) that $w=a_1b_1\ldots a_nb_n$ where $a_i\in A-\{e\}$
and $b_i\in B-\{e\}$. Now on setting $a_i'=\alpha^k(a_i)$ which is in $A$
and $b_i'=\alpha^k(b_i)$ in $B$, we have that $w$ and
$a_1'b_1'\ldots a_n'b_n'$ are conjugate, with both words cyclically reduced
in a free group, meaning that there must be $j$ so that (taking subscripts
modulo $n$) $a_i'=a_{i+j}$ and $b_i'=b_{i+j}$. But then $\alpha^{kn}(a_i)
=(a_i)$, giving a $\z\times\z$ subgroup in $A\rtimes_\alpha\z$.
\end{proof}

Note that in this case if $G=(A*B)\rtimes_\alpha\z$ then both $G_r$ and
$G_q$ are isomorphic to 
the original group $A\rtimes_\alpha\z$.

We now need to find explicit examples of word hyperbolic groups of the
form $F_n\rtimes_\alpha\z$.
To achieve this we use results in \cite{grst} and \cite{st}. An automorphism
$\alpha$ of $F_n$ gives rise to an automorphism of the abelianisation
$\z^n$ of $F_n$ which is well defined for outer automorphisms, as is the
definition of reducibility of $\alpha$.
In the former paper Corollary 2.6 states that if we have an outer automorphism
$O$ of $F_n$ such that the characteristic polynomial of the induced
automorphism of $\z^n$ is a PV-polynomial (which means that the polynomial
is monic and has exactly one root with
modulus greater than one (counted with multiplicity)
and no roots on the unit circle) then $O^k$ is
irreducible for $k\geq 1$. They then obtain as Corollary 2.8 that if
$\alpha$ is an automorphism of $F_n$ for $n\geq 3$ where the corresponding
characteristic polynomial is a PV-polynomial then $\alpha(w)=w$ implies
that $w=e$. Moreover $\alpha$ can have no periodic conjugacy classes
because $\alpha^k$ will also have the same property, and if $\alpha$
(or $\alpha^k$) sends $w$ to a conjugate then we multiply $\alpha$
(or $\alpha^k$) by an inner automorphism which does not change the
characteristic polynomial.

The latter Corollary is proved by applying \cite{bh} Theorem 4.1 which
states that if $O^k$ is irreducible for $k\geq 1$ and $O$ fixes a
conjugacy class then $O$ is geometrically realised by a pseudo-Anosov
homeomorphism of a compact surface with one boundary component. But
\cite{st} shows that automorphisms of $F_n$ with a characteristic polynomial
that is a PV-polynomial are not geometrically realisable for $n\geq 3$
by considering the eigenvalues (note that $F_2\rtimes_\alpha\z$ is never
word hyperbolic because the commutator of the generators gives rise to
a conjugacy class that is fixed or of period 2). An example is given:
$G_0=F_3\rtimes_\alpha\z$ where $\alpha(x)=y,\alpha(y)=z,\alpha(z)=xy$.

\begin{co}
The $F_6$-by-$\z$  group
\begin{eqnarray*}
G=\langle t,a,b,c,x,y,z&|& tat^{-1}=b,tbt^{-1}=c,tct^{-1}=ab,\\
&&txt^{-1}=y,tyt^{-1}=z,
tzt^{-1}=xy\rangle\end{eqnarray*}
is word hyperbolic and large.
\end{co}
\begin{proof}
That $G$ is word hyperbolic follows from the facts quoted above: namely
$G_0=F_3\rtimes_\alpha\z$ has no periodic conjugacy classes so Lemma 2.3
shows that $G$ also has no periodic conjugacy classes thus
is word hyperbolic. That $G$ is large follows from applying
the proof of Theorem 2.1 to $G$ with $G_r=G_q=G_0$, and from the output
of a computer. Inputting the presentation of $G_0$ into MAGMA and asking
for the abelianisation of its low index subgroups, we find (after a bit
of a wait, although it is more quickly checked) that $G_0$ has an index
14 subgroup (with generators $x,y,z^2,zxz^{-1},zyz^{-1},zt^{-7}$) with
abelianisation $C_2\times C_4\times\z\times\z$.
\end{proof}

We also note that $G$ is the double cover of the automorphism which
sends in turn $a$ to $x$ to $b$ to $y$ to $c$ to $z$ to $ab$. Here all
other generators except $t$ and $a$ can be eliminated to get the
2-generator 1-relator $F_6$-by-$\z$ word hyperbolic large group
$\langle t,a|t^6at^{-4}a^{-1}t^{-2}a^{-1}\rangle$.

\section{LERF groups of deficiency 1}

Not all groups with a deficiency 1 presentation can be large, as evidenced
by the Baumslag-Solitar groups $BS(m,n)=\langle a,t|ta^mt^{-1}=a^n\rangle$ for
$m,n$ non-zero integers, where we can take without loss of generality
$m>0$ and $|n|\geq m$. We have that $BS(m,n)$ is large if and only if
$m$ and $n$ are not coprime. For $m=1$ we have a soluble group (and it
is known that a virtually soluble group of deficiency 1 must be isomorphic
to $BS(1,n)$ or to $\z$) 
but otherwise $BS(m,n)$ contains a non-abelian subgroup, thus
it is not true that Baumslag-Solitar groups are either virtually soluble or
large. However most of these groups are not residually finite (for this we
require $m=1$ or $m=|n|$) and so if we stick to residually finite 
Baumslag-Solitar groups, we do have this dichotomy. Thus we could consider
residually finite groups of deficiency 1: in fact we know of no example
of such a group which is not virtually soluble but not large.

In order to make progress we impose an even tighter condition on our
deficiency 1 group $G$, which is that it is LERF (locally extended
residually finite, also known as subgroup separable). This means that
every finitely generated subgroup is an intersection of finite index
subgroups. The big advantage of this is a result of Lubotzky in
\cite{lub} that if a LERF group $G$ can be written as an HNN extension
$H*_\phi$ which is non-ascending, that is both the domain $A$ of $\phi$
and the image $B$ are strictly contained in $H$, then $G$ is large if
$A$ is finitely generated.
This is because $G$ surjects to an HNN extension of a finite group which
is virtually free, and the LERF property applied to $A$
means that this HNN extension
will also be non-ascending, thus it is virtually free but not virtually
cyclic.

\begin{thm}
If $G$ is LERF and of deficiency 1 then either $G$ is large or $G$ is
of the form $F_n\rtimes\z$.
\end{thm}
\begin{proof}
As there must be a surjective homomorphism $\chi$ from $G$ to $\z$, we
have by \cite{bi} that the finitely presented group $G$ is an HNN
extension of a finitely generated group $H$ with finitely generated
associated subgroups $A$ and $B$. We then have that if this HNN
extension is not ascending then $G$ is large by the above. But if say $H=A$
then $G=\langle t,H\rangle$ where $t$ is the stable letter 
and we have $tHt^{-1}
\subseteq H$. By a result of Blass and P.\,M.\,Neumann, a LERF group
cannot have a finitely generated subgroup conjugate to a proper subgroup
of itself so $tHt^{-1}=H$ and we conclude that $H\unlhd G$ with
$G=H\rtimes\z$. Now we can use a recent result of Kochloukova. We have
by \cite{hil} Theorem 6
that if $G$ has deficiency 1 and is an ascending HNN
extension of a finitely generated group then $G$ has geometric (hence
cohomological) dimension at most 2. But Theorem 3 in \cite{ko} is as
follows: Let $G$ be a non-trivial group with a finite K($G$,1) CW-complex
of dimension $n$ with Euler characteristic 0. Suppose that $N$ is a normal
subgroup of $G$ containing $G'$ which is of homological type $FP_{n-1}$
and $G/N$ is cyclic-by-finite. Then $N$ is of type $FP_n$. Consequently
we conclude by putting $H$ equal to $N$ that $H$ is of type $FP_2$, but
a result of Bieri gives us that the cohomological dimension of $H$ is 1
and thus $H$ is free.
\end{proof}
\newline
It should be noted that groups of the form $F_n\rtimes\z$ are not
necessarily LERF.

Consequently if we had that $F_n\rtimes\z$ was large for $n\geq 2$ we
would have the strongest possible result: either a deficiency 1 LERF
group is large or it is $\z$, $\z\times\z$ or the Klein bottle group
$BS(1,-1)$. However we can move in two directions from this result. The
first is to relax the property of largeness to that of $G$ being 
SQ-universal (every countable group is a subgroup of a quotient of $G$)
which in turn is stronger than containing the free group $F_2$. As
groups of the form $F_n\rtimes\z$ for $n\geq 2$ are either large by \cite{me}
or hyperbolic, hence SQ-universal by \cite{ol}, we immediately get:
\begin{co}
A LERF group of deficiency 1 is either SQ-universal or is one of
$\z$, $\z\times\z$, $BS(1,-1)$.
\end{co}
A conjecture of P.\,M.\,Neumann from 1973 is that there is a dichotomy 
for any 1-relator group: either it is SQ-universal or it is a soluble
Baumslag-Solitar group or $\z$.
Of course if there are at least 3 generators then
later it was shown that we have largeness but the conjecture remains
open in the 2-generator 1-relator case. Corollary 7.5 in \cite{me} showed
this to be true if the group is LERF, so in light of Corollary 3.2 we wonder
if P.\,M.\,Neumann's conjecture extends to all deficiency 1 groups, or
(more cautiously) to all residually finite deficiency 1 groups.

The other improvement that we can make to Theorem 3.1 is to give further
conditions guaranteeing largeness. It might not be a surprise in light
of Section 2 that once again it comes down to requiring first
Betti number at least two.
\begin{thm}
If $G$ is LERF, of deficiency 1 and has a finite index subgroup with
first Betti number at least 2 then $G$ is large or $\z\times\z$ or the
Klein bottle group.
\end{thm}
\begin{proof}
We assume that $\beta_1(G)=b\geq 2$ and is
LERF (which is a property preserved by all subgroups). We consider the
BNS (Bieri-Neumann-Strebel) invariant of $G$ which is an open subset of
$S^{b-1}$ and which gives information on the finite generation of kernels
of non-trivial homomorphisms from $G$ to $\z$. By regarding the rationally
defined points of $S^{b-1}$ as equivalence classes of homomorphisms
$\chi,\chi'$ from $G$ to $\z$ according to the relation $\chi=q\chi'$ for
$q\in\q$ and $q>0$, we have that $[\chi]$ is in $\Sigma$ if and only if
$G$ can be expressed as an ascending HNN extension $\langle t,H\rangle$
with associated homomorphism $\chi$ (that is $\chi(t)=1$ and $\chi(H)=0$)
and with $H$ finitely generated. Moreover ker $\chi$ is finitely
generated if and only if $[\chi]$ and $[-\chi]$ are both in $\Sigma$.
Now for a LERF group $G$, if there exists $\chi$ with neither of $[\pm\chi]$
in $\Sigma$ then $G$ is large, as in Theorem 3.1. Otherwise one of
$[\pm\chi]$ is in $\Sigma$ for every homomorphism $\chi$ from $G$ onto
$\z$. But as the LERF condition means that we cannot have $G$ equal to a
strictly ascending HNN extension $\langle t,H\rangle$ where $H$ is finitely
generated, we must have both $[\pm\chi]\in\Sigma$ for all possible $\chi$.

Now we use a result of Dunfield in \cite{dun}. One can regard the 
(multivariable) Alexander polynomial $\Delta_G(t_1,\ldots ,t_b)$ as a
finite set of lattice points in $\z^b$ labelled with a non-zero integer
by taking the monomials in $\Delta_G$ with non-zero coefficient.
We can then form the Newton polytope $N(\Delta_G)\subseteq
\mathbb R^b$ which is the
convex hull of these points. Then \cite{dun} Theorem 5.1 states that if
$D(\Delta_G)$ is the dual of $N(\Delta_G)$ in $\mathbb R^b$ (so that faces of
dimension $i$ become faces of dimension $b-i-1$) and $F_1,\ldots ,F_k$
are the $(b-1)$-dimensional faces of $D(\Delta_G)$ whose corresponding
vertices of $N(\Delta_G)$ have coefficient $\pm 1$ then $\Sigma$ is
contained in the projection of the interiors of the $F_i$ to $S^{b-1}$.

So far we have not used the fact that $b\geq 2$. Now when $b=1$ we get
$S^0=\{\pm 1\}$ and the result above is saying that if $\chi$ is the
unique surjective homomorphism (up to sign) from $G$ to $\z$ then ker
$\chi$ being finitely generated implies that the highest and lowest terms
of $\Delta_G(t)$ are monic. However if $b\geq 2$ then the fact that
$[\chi]\in\Sigma^{b-1}$ for all homomorphisms $\chi$ implies that
$\Delta_G=1$. This is because if there are $n\geq 2$ vertices of 
$N(\Delta_G)$ then we have $(b-1)$-dimensional faces $F_1,\ldots ,F_n$
of $D(\Delta_G)$, and so when we project the interiors of those faces
which are obtained
from the $\pm 1$ coefficients, we do not cover all of $S^{b-1}$
because we miss the lower dimensional faces where pairs of elements of
$F_1,\ldots ,F_n$ meet. Moreover there will be rationally defined points
which are not covered and therefore homomorphisms $\chi$ with $[\chi]$
not in $\Sigma$.

Hence if $G$ is LERF and has $\beta_1(G)\geq 2$ then the only way that
$G$ fails to be large is if $\Delta_G=1$, with ker $\chi$ being finitely
generated for all homomorphisms $\chi$ onto $\z$.
However by \cite{me} Theorem 3.1 we have for deficiency 1 groups that
\[\Delta_{G,\chi}(t)=(t-1)\Delta_G(t_1^{n_1},\ldots ,t_b^{n_b})\]
where $n_i$ is the image under $\chi$ of any element in $G$ which projects
to $t_i$ under the natural homomorphism from $G$ to its free abelianisation
$\z^b$.
Consequently if $\Delta_G=1$ then for any $\chi$ we have 
$\Delta_{G,\chi}$ equal to $t-1$. 
But the degree of $\Delta_{G,\chi}$ is the dimension of the 
$\q$-vector space $H_1(\mbox{ker }\chi;\q)
=H_1(\mbox{ker }\chi;\z)\otimes_\z\q$, for which we write 
$\beta_1(\mbox{ker }\chi;\q)$, and from \cite{ko} Theorem 3 we have that
ker $\chi$ is free, so it must be free of rank 1 and therefore $G=\z\times\z$.

%Moreover by \cite{me} Theorem 3.1 we have for deficiency 1 groups that
%\[\Delta_{G,\chi}(t)=(t-1)\Delta_G(t_1^{n_1},\ldots ,t_b^{n_b})\]
%where $n_i$ is the image under $\chi$ of any element in $G$ which projects
%to $t_i$ under the natural homomorphism from $G$ to its free abelianisation
%$\z^b$. Consequently if $\Delta_G=1$ then for any $\chi$ we have 
%$\Delta_{G,\chi}$ and hence $\Delta^p_{G,\chi}$ equal to $t-1$. Thus
%for $K=\mbox{ker }\chi$ we have $\beta_1(K;\z/p\z)=1$. As we have for an
%abelian group $A$ that $\beta_1(A;\z/p\z)\leq\beta_1(A/H;\z/p\z)
%+\beta_1(H;\z/p\z)$,
%we have on putting $A=G/G'$ and $H=K/G'$ that
%\begin{equation}
%\beta_1(G;\z/p\z)\leq\beta_1(\z;\z/p\z)+\beta_1(K/G';\z/p\z)\leq 1
%+\beta_1(K;\z/p\z)
%\end{equation}
%which is 2 in this case. Consequently we are already done if $b\geq 3$.
%If $b=2$ then we are done as soon as we find a finite index subgroup $H$
%of $G$ with $\beta_1(H;\z/p\z)\geq 3$. But $\beta_1(H)\geq 2$ for all
%$H\leq_f G$ so we fail only if $H/H'$ is always $\z\times\z$. But then
%by \cite{me2} Proposition 3.4 either $G=\z\times\z$ or $G$ has the
%property that $G'$ is non-trivial but $G'\leq H$ for all $H\leq_f G$.
%The latter case
%means that $G$ completely fails to be residually finite, so it
%certainly is not LERF. 
Finally if $G$
has deficiency 1 and has a finite index subgroup $H=\z\times\z$ then $G$
is either $\z\times\z$ or the Klein bottle group.
\end{proof}
\newline
We will see in Corollary 4.6 that this theorem is not true if the LERF
condition is removed. If it is weakened to residually finite then this
is unknown.

We can also finish off the proof without needing \cite{ko} Theorem 3
by considering the Alexander polynomial $\Delta^p_{G,\chi}$ over $\z/p\z$.
Although it is not true for a general finite presentation, for the 
deficiency 1 case this is just the Alexander polynomial $\Delta_{G,\chi}$
over $\z$ reduced modulo $p$. Consequently we must have 
$\beta_1(\mbox{ker }\chi;\z/p\z)=1$ and so $\beta_1(G;\z/p\z)\leq 2$.
But if this is true for all finite index subgroups $H$ of $G$ then $H/H'
=\z\times\z$ which cannot happen if $G$ is residually finite by 
\cite{me} Proposition 3.4 unless $G=\z\times\z$.

As for attempting to apply the above techniques to deficiency 1 groups
which are not LERF, we can squeeze out results on when a deficiency 1
group contains a non-abelian free group (or equivalently contains $F_2$).
In \cite{w252} J.\,S.\,Wilson conjectures that if a finitely presented
group $G$ is such that def$(G)+d^2/4-d>0$, where $d$ is the minimum
number of generators for $G/G'$ (and assumed to be at least 2), then
$G$ contains a non-abelian free group. (Bartholdi has recently proved in
\cite{bth} that such a group is non-amenable.) 

For def$(G)=1$ this was
established in \cite{how}, as Corollary 2.4 in that paper states that
if $G$ is of deficiency 1 and $N\unlhd G$ is such that $G/N$ is non-trivial
and free abelian with $N/N'\otimes_\z \mathbb F$ 
having dimension at least 2 for
some field $\mathbb F$, then $G$ contains $F_2$. 
Therefore on putting $N=\mbox{ker }\chi$ for any $\chi$, we have that
$\beta_1(G;\z/p\z)$ is at least 3 for some prime $p$ which implies that
$\beta_1(N;\z/p\z)\geq 2$. 

We also have
\begin{co}
If $G$ has deficiency 1 and virtual first Betti number at least 2 then $G$
contains $F_2$ or is $\z\times\z$ or the Klein bottle group.
\end{co}
\begin{proof}
By \cite{bns} Theorem D, if a finitely presented group $G$ does not contain
$F_2$ but $\beta_1(G)\geq 2$ then there exists $K=\mbox{ker }\chi$ which
is finitely generated and such that $G/K\cong\z$. Consequently by \cite{ko}
Theorem 3 mentioned above, we have that $K$ is of type
$FP_2$ and hence free as before. Thus we are done unless $K$ is free of
rank 0 (but then $G=\z$ so its virtual first Betti number is 1) or free of
rank 1, giving the two exceptions.
\end{proof}

In fact it is widely believed that if $G$ has deficiency 1 then it contains
$F_2$ unless $G=BS(1,n)$. It is clear that the only case left is when
$\beta_1(G)=1$ with $1\leq d(G/G')\leq 2$ and $G$ is a strictly ascending
HNN extension, so that $S^0$ contains two points with $\Sigma$ one of
them. The conjecture at the end of Section 2 in \cite{binw} is that in this
case $G$ is a strictly ascending HNN extension with base a finitely
generated free group. As it is proved in this paper that for $G$ a
finitely presented group $\Sigma=\Sigma^1(G;\z)$ is equal to the 
higher dimensional invariant $\Sigma^2(G;\z)$, we would be done if we knew
that $[\chi]\in\Sigma^2(G;\z)$ for a rationally defined $\chi$ implies
that $G$ is an ascending HNN extension over a base group of type $FP_2$
with associated homomorphism $\chi$, just as is the case for 
$[\chi]\in\Sigma$ and with base of type $FP_1$, i.e. finitely generated.
In the special case of 2-generator 1-relator groups the conjecture is
known: indeed in \cite{ls} Chapter II Section 5
we have a ``Tits alternative'' which says
that a subgroup of a 1-relator group either contains $F_2$
or is soluble (whereupon it is locally cyclic, $BS(1,n)$ or infinite
dihedral). However we will not get such an alternative for arbitrary
subgroups of deficiency 1 groups, or even arbitrary finitely generated
subgroups:
\begin{ex}
\end{ex}
Thompson's group $T$ has a finite presentation of two generators and
two relators, but has the unusual property that it does not contain $F_2$
nor is it virtually soluble. The group $G=T*\z$ has deficiency at least one,
and in fact it is exactly 1 as we can use Philip Hall's
inequality that def$(G)\leq
\beta_1(G)-d(H_2(G;\z))$ (see \cite{rob} 14.1.5). 
We have $H_n(T;\z)=\z\times\z$ 
for all $n\geq 1$
by \cite{brg} Theorem 7.1, so that by using the Meier-Vietoris sequence
for a free product we get $\beta_1(G)=3$ and $H_2(G;\z)=\z\times\z$.

Therefore any formulation of a possible Tits alternative for subgroups
of deficiency 1 groups will need to avoid freely decomposable examples.
  
\section{2-generator 1-relator groups of height 1}

So far our attempts to prove that various deficiency 1 groups are large
have needed some hypothesis on the group, such as being free-by-cyclic
or LERF. In this section we concentrate on 2-generator 1-relator groups
and look for conditions where we can conclude that such a group is large
using information obtained directly from a given presentation. We fall
into two very different cases: a group $G$ of the form $\langle x,y|r(x,y)
\rangle$ has $\beta_1(G)$ equal to either 1 or 2, the latter occurring
exactly when $r\in F_2'$. First let us concentrate on the former case, where
by making a change of free basis for the group $F_2$
we can assume that the presentation is of the form $\langle a,t|w(a,t)
\rangle$ where the word $w$ has exponent sum 0 in $t$ and is cyclically
reduced. Consequently we can define the height of such words: we can
rewrite $w$ as a word in $a_i=t^iat^{-i}$ for $i\in\z$ but only finitely
many letters $a_i$ will actually appear in $w$ (this is
sometimes referred to as Moldavanski\u{i} rewriting).
If $a_m$ is the smallest
and $a_M$ the largest letter to appear then the height of $w$ is defined
to be $M-m\geq 0$, which is invariant under cyclic permutations and
taking inverses.

If the height is zero then we can only have $w=a^i$ so $G$ is large (or
$\z$ for $|i|=1$). We now consider height one words where without loss
of generality we can assume that
\begin{equation}
w=ta^{i_1}t^{-1}a^{i_2}\ldots ta^{i_{2k-1}}t^{-1}a^{i_{2k}}.
\end{equation}
Note that for $\chi(t)=1,\chi(a)=0$ we have $\Delta_{G,\chi}(t)=
(i_1+\dots +i_{2k-1})t+i_2+\ldots +i_{2k}$ and for $\beta_1(G)=1$ this is
the Alexander polynomial $\Delta_G$.

If $\beta_1(G)=2$ then the exponent sum of $x$ and $y$ is zero, so we
can talk about the height of either letter. Moreover a change of basis
will preserve this property but could well vary the heights. The results
that follow are also satisfied by groups $G$ with $\beta_1(G)=2$, where
we interpret a height 1 word as there exists a free basis where the word
has height 1 with respect to one of the letters.
\begin{thm}
If $G=\langle a,t|w\rangle$ where $w$ is a height 1 word then either
$G$ is large or the finite residual $R_G=G''$, in which case $G/R_G$ is
metabelian and all finite images of $G$ are metacyclic.
\end{thm}
\begin{proof}
By writing $w$ in terms of $a_0,a_1$ we see that $G$ is
an HNN extension of the 1-relator group $H=\langle a_0,a_1|w(a_0,a_1)
\rangle$ with associated cyclic subgroups $A_0=\langle a_0\rangle$ and $
A_1=\langle a_1\rangle$. We can now use Lubotzky's result as mentioned in
the previous section: although it appears to require that $G$ is LERF,
all we need to conclude largeness for $G$ is
the existence of a homomorphism of $G$ onto a finite group with
$\theta(H)\neq\theta(A_0)$. But $\theta(A_0)$ is cyclic so if there is
no such homomorphism then, as $t$ conjugates $a_0$ into $H$, we have
that $\theta(H)$ is normal in $\theta(G)$
and cyclic, with $\theta(G)/\theta(H)\cong
\langle\theta(t)\rangle$ cyclic as well. This means that all finite images
of $G$ are metacyclic  and if $x\in G''$ but $x\notin R_G$, we would take
a homomorphism $\theta$ from $G$ onto a finite group $F$ with $\theta(x)
\neq e$ so that $F''$ is non-trivial, but this is a contradiction.
Consequently $G''\leq R_G$ but Philip Hall's result that a finitely generated
metabelian group is residually finite implies that $R_G\leq G''$.
\end{proof}

Moreover note that the normal closure of the element $a$ in $G/R_G$,
which is generated by elements of the form $t^iat^{-i}$ for $i\in\z$, is
abelian because if there existed $[t^iat^{-i},a]\notin R_G$ then we could
take a finite image of $G$ in which this was also non-trivial, but this
contradicts the fact that the cyclic group $\langle a\rangle$ is normal
in every finite image.

This gives us a strict dichotomy in the behaviour of height 1 presentations.
However the result does not tell us how to work out in which category
a given presentation falls. Luckily there is a straightforward way of
determining largeness once a major result of Zelmanov is used.

\begin{co}
If $G=\langle a,t|w\rangle$ where $w$ is a height 1 word then $G$ is large
if and only if there exists $H\leq_f G$ with $d(H/H')\geq 3$, where $d$
is the minimum number of generators of a finitely generated group.
\end{co}
\begin{proof} The only if direction is clear. \cite{zel} states that if we have
a finite presentation $P$ of a group $G$ with $\mbox{def}(P)+d^2/4-d>0$
where $d=d(G/G')$ is at least 2 then there exists a prime $p$ such that
the pro-$p$ completion of $G$ contains a non-abelian free pro-$p$ group.
In our case the deficiency is 1 so this is true if $d\geq 3$. But if $G''=R_G$
then all finite images $F$ of $G$ have $F''$ trivial, so this is also
true for the profinite or pro-$p$ completion of $G$. Consequently any
pro-$p$ completion is soluble and cannot contain a non-abelian free
group.

However $G$ is 2-generated so we cannot have $d(G/G')\geq 3$ in any case.
But for $H\leq_f G$ we still have deficiency 1, with $H''\leq G''$ and
$R_G=R_H$, so if $G''=R_G$ and $d(H/H')\geq 3$ we have a contradiction.
\end{proof}

We now give a summary of height 1 group presentations that
have appeared in the literature.
If we define the length of height 1 words to be $k$ in (2) then it is clear
that length 1 words are just the standard presentations of the
Baumslag-Solitar groups $BS(m,n)$. In \cite{bamo} Baumslag introduced a
family $C(m,n)$ of groups with presentations
\[\langle a,t|(tat^{-1}a^m(tat^{-1})^{-1}=a^n\rangle\]
which generalises the group $C(1,2)$ that first appeared in \cite{bam}
with the comment that they are 1-relator groups which are
as far from being residually finite as possible. For our purposes
we will interpret
a finitely generated group $G$ as being ``as far from being residually
finite as possible'' in three ways: the first is that $G$ (is infinite
but) has no proper finite index subgroups whatsoever, or equivalently
$G=R_G$. However this will never arise with deficiency 1 groups as $G/G'$
is infinite.

The next strictest interpretation is that ($G$ is non-abelian but) $G'=R_G$,
or equivalently every finite image is abelian. 
To be succinct, we will say that $G$ is {\bf proabelian} 
because this is equivalent
to the profinite completion of $G$ being proabelian (or abelian).
This property of finitely
generated groups was examined in \cite{me}. 
A finite index subgroup of a proabelian group is also proabelian and a
finitely presented proabelian group $G$ (with $\beta_1(G)>0$)
has $\Delta_G=1$ (which can be seen because the Alexander polynomial of $G$
is the same as that of $G/R_G$ and this will be 1 for finitely generated
infinite abelian groups). If $G$ is a deficiency 1 group with the 
abelianisation $G/G'$ equal to $\z\times T$ then $G$ being proabelian
and $\beta_1(G)=1$
implies that $T$ must be trivial because here $|\Delta_G(1)|=|T|$. We cannot
have $d(G/G')\geq 3$ from above, so that if $G$ is proabelian and 
$\beta_1(G)\geq 2$ then $G/G'=\z\times\z$. Moreover if $H/H'=\z$ for all
$H\leq_f G$ (or $H/H'$ is $\z\times\z$ throughout) then $G$ is proabelian
by \cite{me} Proposition 3.4. In particular if $G$ has deficiency 1 with
$\beta_1(G)=2$ then either $G$ is proabelian (so that $G/R_G=\z\times\z$)
or a pro-$p$ completion of $G$ contains a non-abelian free pro-$p$ group
(so that $G/R_G$ is ``big'').

The first example of a proabelian finitely generated group 
that is given by a 1-relator presentation is
$C(1,2)$; indeed every finite
image is cyclic as the abelianisation of $C(1,2)$ is $\z$. 

Finally we could also look at when ($G$ is not metabelian but) $G''=R_G$,
or equivalently every finite image is metabelian, for which we will 
similarly write that $G$ is {\bf prometabelian}. 
As we have $R_G\leq G''$ for any
finitely generated group (although not necessarily $R_G\leq G'''$) this is 
also a
natural concept, with a prometabelian group $G$ being proabelian if and
only if $G'=G''$. Moreover it is also fair to say that finitely
generated proabelian and prometabelian groups are 
``as far from being large as possible''.

If we have a height 1 group $G=\langle a,t|w(a_0,a_1)\rangle$ which is not
large then by the comment after Theorem 4.1 we have that $a_0$ commutes with
$a_1$ in $G/R_G$. In particular if the exponent sum of $a_0$ in $w(a_0,a_1)$
is $d$ and that of $a_1$ is $c$ then $G/G''$ has the presentation
\[\langle a,t|ta^ct^{-1}=a^{-d},[t^iat^{-i},a]\mbox{ for }i\in\z\rangle\]
which is the same as for $BS(c,-d)/BS(c,-d)''$. Thus any result which depends
purely on the finite images of $BS(m,n)$ for $m$ and $n$ coprime applies
identically for any non-large group $G$ of height 1 (and if $c$ and $d$
are not coprime then $G$ is large anyway).
In particular the formulae for the number of finite index subgroups of
$BS(m,n)$ in \cite{gel} and the number of finite index normal subgroups
in \cite{mebs} apply equally for $G$. However there is a potential
problem here: we do not know whether these groups are merely
nonstandard presentations of $BS(m,n)$. For instance on putting $b=a^2$
and substituting in for the standard presentation of $BS(2,3)$ we have
$a=tbt^{-1}b^{-1}$ and so we obtain
an alternative height 1 presentation
\[\langle b,t|tbt^{-1}b^{-1}tbt^{-1}b^{-2}\rangle.\]
The problem is a lack of invariants which are able to distinguish that two
given presentations do not define isomorphic groups. We see in these cases
that any such invariant which is calculated using information obtained solely
from the finite images of a group is here doomed to failure.

As for 1-relator proabelian groups that are not abelian, we can only
have presentations with 2 generators. We have already mentioned $C(1,2)$
above and in \cite{ep} the example $C(2,3)$ was introduced where it was
shown that for all $H\leq_fC(2,3)$ we have $H/H'=\z$
so $C(2,3)$ is proabelian by \cite{me}. 
Moreover if we regard $C(2,3)$
as an HNN extension of $BS(2,3)$ by adding the stable letter $s$ with
$sas^{-1}=t$ then we can iterate this process indefinitely and keep on
obtaining new 1-relator height 1 group presentations.
This HNN extension when applied
to $C(1,2)$ was also used in \cite{mek} to give an example of a higher
dimensional knot whose infinite cyclic cover is not simply connected but
which has no proper finite covers. The group $C(1,2)$ (which is
sometimes referred to as the Baumslag-Gersten group)
also appears in
\cite{geris181} where it is shown that its Dehn function grows faster
than every iterated exponential, and again in 
\cite{bru} where it is shown that it is isomorphic to
\[\langle a,t|(ta^{2^k}t^{-1})a(ta^{2^k}t^{-1})^{-1}=a^2\rangle\]
for any $k\geq 0$. The paper \cite{brunclss} considers a more general
family containing the groups $C(m,n)$ where the conjugating element
$tat^{-1}$ can be $ta^kt^{-1}$. These were used in \cite{bormold} to give
a pair of non-isomorphic 1-relator groups, each of which is
a homomorphic image of the other.
Also the examples
\[\langle a,t|(t^kat^{-k})a(t^kat^{-k})^{-1}=a^2\rangle\]
of proabelian groups for $k\geq 1$ are mentioned in \cite{mdsb}: note this
implies that a proabelian 1-relator group can have arbitrary height.

We can incorporate all of these examples in the next result. Although the
proof is really that of Higman in \cite{hig} when giving the first example
of an infinite
finitely generated simple group, the aim here is to express it in as
general terms as possible.
\begin{thm}
Suppose that $g$ and $h$ are elements of a finite group $G$ with $h^kg^m
h^{-k}=g^n$ for integers $k,m,n$ where $|m-n|=1$. Suppose further that
$g$ and $h$ have equal orders. Then this order is always 1.
\end{thm}
\begin{proof}
If $r$ is the order of $g$ and $h$ then by considering $h^{rk}g^{m^r}h^{-rk}$
which is equal to $g^{n^r}$, we must have $r$ dividing $m^r-n^r$. So if
$r\neq 1$ and $p$ is the smallest prime dividing $r$, we have $(mn^{-1})^r
\equiv 1\mbox{ mod }p$ where we can assume without loss of generality that
$n$ and $p$ are coprime. Thus the order of $mn^{-1}$ divides $r$ and $p-1$
so $p$ divides $m-n=\pm 1$.
\end{proof}
\begin{co}
Suppose $H=\langle a,t|w(a,t)\rangle$ where $w$ has exponent sum 0 in $t$
and is of height 1. Then the HNN extension $G=\langle
H,s\rangle$, where $sas^{-1}=t$, is large if $H/H'\neq\z$ and proabelian
if $H/H'=\z$ but $H$ is not large.
\end{co}
\begin{proof}
$G=\langle a,s|w(a,sas^{-1})\rangle$ is also a 1-relator group of height
1 and if the Alexander polynomial of $H$ (with respect to the homomorphism
$\chi(t)=1,\chi(a)=0$) is $f(t)=ct+d$ then that for $G$ is $c+d$. Thus $G$
is large by Howie's result as the mod $p$ Alexander polynomial is zero
unless $|c+d|=1$. But $|f(1)|$ is the order of $T$, where
$H/H'=\z\times T$.

If $|T|=1$ and $H$ is not large
then suppose the finite group $F=\theta(G)$. We have that in $F$
the relation $ta^ct^{-1}=a^{-d}$ holds because $\theta(H)$ is a finite image
of $H$ and $H''=R_H$.
But $t$ and $a$ are conjugate in $F$ via the image of $s$,
so we can use Theorem 4.3 to conclude that $\theta(t)=\theta(a)=e$.
\end{proof}

Note that this means the group $C(m,n)$ above is not large if and only
if $|m-n|=1$. This is in disagreement with \cite{bamo} Theorem 4 where
part 2 states (but proofs are not given) that $C(m,n)$ is proabelian if
$m$ and $n$ are distinct primes. However in Section 3 of that paper
it is mentioned that
if $m=n+1$ then $C(m,n)$ is proabelian so we believe that is what was meant
in the theorem.

Corollary 4.2 would allow us, using a computer, to gather statistics on
what proportions of height 1 groups $G$ are large by searching through
finite index subgroups $H$ of $G$ until $d(H/H')\geq 3$ (or until we
give up on $G$), thus giving a lower bound for the number of large groups.
We propose to return to this but content ourselves for now by noting that
the computer tells us that the group $\langle a,t|ta^2t^{-1}a^{-1}ta^{-1}
t^{-1}a^{-1}\rangle$
is large but has Alexander polynomial $t-2$ which is the same 
as $BS(1,2)$, and forming the HNN extension of this group using $sas^{-1}=t$
results in a height 1 group which is large but with Alexander polynomial
equal to 1.
\begin{co}
If $G_0=\langle a,t|w\rangle$ is proabelian then $G_1=\langle a,t|v^kw^mv^{-k}
=w^n\rangle$ is proabelian, where $v$ is a conjugate element of $w$ in
$F_2$ and $|m-n|=1$.
\end{co}
\begin{proof}
Given any finite image $F$ of $G_1$, we set $h$ equal to the image of $v$ and
$g$ equal to the image of $w$ and apply Theorem 4.3. We conclude that the
image of $w$ is trivial in $F$ so $F$ is a finite image of $G_0$ and hence
is abelian.
\end{proof}

Thus we see how to create lots of proabelian 1-relator groups by starting
with $w=a$ (or indeed $w$ equal to any element of a free basis for $F_2$)
and then iterating in Corollary 4.5. We can also give an example of a
1-relator group $G$ with $\beta_1(G)=2$ that is proabelian,
the existence of which
was not previously known. Indeed the only known example up to now
of such a group $G$ which is not residually finite was given by 
Ol'shanski\u{i} in \cite{nypl} but it is not directly constructive.
\begin{co}
The group
\[G=\langle a,t|[a,t][t,a^{-1}][a,t]^{-1}=[t,a^{-1}]^2\rangle\]
is proabelian (hence not large) with $\beta_1(G)=2$ but is not equal
to $\z\times\z$ (hence is not residually finite).
\end{co}
\begin{proof}
On applying Corollary 4.5 with $w=[t,a^{-1}]$, so that $G_0=\z\times\z$ and
$v=[a,t]$, we obtain the above presentation for $G=G_1$. But it is known
by \cite{mks} Section 4.4
that a 2-generator 1-relator group $\langle a,t|r\rangle$ is equal
to $\z\times\z$ if and only if $r=[a,t]^{\pm 1}$ or a cyclic conjugate
when $r$ is cyclically reduced.
\end{proof}

We note that the relation above is of height 1 with respect to $t$.
However there are also groups with first Betti number 2 and height 1
which we can prove are large:
\begin{co} If $G=\langle a,t|w(a,t)\rangle$ where $\beta_1(G)=2$ and $w$
is a height 1 word of the form in (2) then $G$ is large if
$|i_1+i_3+\ldots +i_{2k-1}|\neq \pm 1$ and is proabelian or large
otherwise.
\end{co}
\begin{proof}
We have that $\beta_1(G)=2$ implies $i_1+i_2+\ldots +i_{2k}=0$ and so
$\Delta_{G,\chi}(t)=n(t-1)$ where $n=i_1+i_3+\ldots +i_{2k-1}$ with 
$\chi(a)=0$, $\chi(t)=1$. Thus if $|n|\neq\pm 1$ we have largeness
because $\Delta_{G,\chi}^p\equiv 0$ modulo a suitable prime $p$. We also
have largeness by Corollary 4.2 unless $H/H'=\z\times\z$ for all $H\leq_f G$
which implies that $G$ is proabelian (or $G=\z\times\z$ if $w=ta^{\pm 1}
t^{-1}a^{\mp 1}$).
\end{proof}

For instance the group $\langle a,b,t|aba^2=b^2,ta^3t^{-1}=b\rangle$
is shown in \cite{brcp} Proposition 19
to be word hyperbolic and of the form $F_6\rtimes_\alpha\z$.
As such groups are residually finite and eliminating $b$ gives a height 1
word with respect to $t$, we can use Corollary 4.7 to obtain another
group satisfying the conditions of Corollary 2.4 without using the computer.

\section{Problem List}

Here we list some problems in Group Theory with particular emphasis on
finitely generated and finitely presented groups. The aim is to put
together questions which may well have appeared somewhere in print but
which are not to be found in the standard problem lists (for instance
the Kourovka notebook \cite{kou}, the New York list \cite{nypl} and Mladen
Bestvina's questions in Geometric Group Theory \cite{bspl}). Also there is
an emphasis on questions which can be stated using only group theoretic
concepts, although it may well be that solutions require topological,
geometric or other techniques. In some cases the credits refer to those
from whom we first heard about the question; we apologise if this is
not the original source. The order of appearance places what should be
the most general questions first and then specialises until we reach the
topics of this paper.\\
\hfill\\
{\bf Finiteness questions}\\
{\bf 1.} (J.\,S.\,Wilson \cite{w252}) Is there a finitely presented
and residually finite group which is neither virtually soluble nor
contains a non-abelian free group?\\
{\bf Notes:} If we replace finitely presented with finitely generated
then certainly a range of examples are known. If however we keep
finitely presented and remove residually finite then the list becomes
much shorter. We know of only four constructions: Thompson's group $F$, the
Houghton groups as in \cite{houg}, 
the finitely presented Grigorchuk group (which
is an ascending HNN extension of the well known Grigorchuk group) and the
non-amenable monsters (which are ascending HNN extensions of finitely
generated infinite groups with finite exponent) in \cite{olsa}. All of
these are far from being residually finite. Of course if we weaken
finitely presented to finitely generated but strengthen residually
finite to linear then there are no examples by the Tits alternative.

A variation might be to strengthen the conditions in other ways
(e.g. LERF, coherent, of type FP) with the aim of proving such an example
does not exist.\\
\hfill\\
The Pr\"ufer rank of a group is the supremum of the minimum number 
of generators over all finitely generated subgroups.\\
{\bf 2.} If a group has infinite Pr\"ufer rank then does it have an
infinitely generated subgroup?\\
{\bf Notes:} This question has interesting implications either way. We
say that a group has max (or is Noetherian) if it and every subgroup
is finitely generated. The question of whether a group with max is
virtually polycyclic was settled in the negative by Ol'shankski\u{i}'s 
construction of Tarski monsters, but these have finite Pr\"ufer rank.
Moreover this property (like max) is preserved by subgroups, quotients
and extensions. Thus a counterexample in question 2 would be a genuinely new
example of a group with max as it would not be obtained from known examples
by these operations.

If however the answer is no then a result of Lubotzky and Mann in \cite{lumn}
says that a finitely generated, residually finite group with finite
Pr\"ufer rank is virtually soluble. Thus a residually finite group with
max would be virtually polycyclic, giving a positive answer to question
31 in the first Kourovka notebook, credited to M.\,I.\,Kargapolov.\\
\hfill\\
{\bf 3.} If $G/N$ and $N$ are both virtually soluble then is $G$?\\
{\bf Notes:} This is true and straightforward with virtually soluble
replaced by soluble.\\
\hfill\\
{\bf 4.} Is there a finitely presented group that is elementary amenable
but not virtually soluble?\\
{\bf Notes:} In \cite{hilea} Hillman and Linnell
give an infinitely generated example
and then use it to obtain a finitely generated example.\\
\hfill\\
{\bf Amenability and properties $(T)$ and $(\tau)$}\\
{\bf 5.} Do all finitely generated (or finitely presented) infinite
amenable groups
have a proper finite index subgroup?\\
{\bf Notes:} The rationals show this is not true for infinitely generated
groups. We have the dichotomy that no infinite discrete group
can have property $(T)$ and be amenable. For property $(\tau)$ (always assumed
here to be with respect to all normal finite index subgroups) we 
have by \cite{luta} Proposition 3.3.7 that a finitely generated amenable
group with infinitely many finite index subgroups
does not have $(\tau)$, but groups with finitely many finite index
subgroups trivially have property $(\tau)$. If the answer is yes then we
regain our dichotomy for finitely generated/presented groups.

An old question asks if there is a finitely generated 
infinite simple group which is amenable. This is
equivalent to having a finitely generated infinite amenable group
with no finite index subgroups, as we can quotient out by a maximal
normal subgroup. Infinitely
generated simple amenable groups certainly exist, for instance the
union $\cup A_n$ of alternating groups. It has also been asked whether
there are finitely presented infinite simple amenable groups, but those
wishing to establish their existence would need to answer the next
question in the negative.\\
\hfill\\
{\bf 6.} Does an infinite finitely presented simple group necessarily
contain a non-abelian free group?\\
{\bf Notes:} This is not true for infinite finitely generated simple
groups but for all known finitely presented examples the answer seems
to be yes.\\
\hfill\\
{\bf 7.} (Lackenby) If a finitely presented group has zero virtual first
Betti number then does it have property $(\tau)$?\\
{\bf Notes:} The Grigorchuk group shows that this is not true for finitely
generated groups as it is amenable with infinitely many finite index subgroups.

We recall that property $(T)$ implies finite generation and zero virtual first
Betti number, as well as implying property $(\tau)$ which also implies
zero virtual first Betti number.\\
\hfill\\
{\bf Subgroup Separability}\\
{\bf 8.} (J.\,S.\,Wilson) Is there a finitely generated group which is
not virtually polycyclic but where every subgroup is the intersection of
finite index subgroups?\\
{\bf Notes:} The paper \cite{jeaw} shows that there are no examples when
the group is hyper-(abelian or finite) which covers being virtually
soluble. As for a finitely presented example, this would be residually
finite and would not contain $F_2$, so would provide a yes answer to
question 1.\\
\hfill\\
{\bf 9.} (Long and Reid \cite{lr} Question 4.5) Is there a finitely generated
infinite group which is LERF and has property $(T)$?\\
\hfill\\
{\bf 10.} Is there a finitely presented infinite group which is LERF and
which has zero virtual first Betti number?\\
{\bf Notes:}  A finitely presented example for Question 9 (or just one
with property $(\tau)$) would of course answer this. However the Grigorchuk
group is a finitely generated infinite
LERF group which has zero virtual first
Betti number.\\
\hfill\\
{\bf Largeness}\\
{\bf 11.} If a finitely generated or finitely presented group has the
biggest possible subgroup growth, that is of strict type $n^n$, then is it
large?\\
\hfill\\
{\bf 12.} Is there an algorithm to determine whether or not the group
given by a finite presentation has a proper finite index subgroup?\\
{\bf Notes:} It is unknown whether there is an algorithm to detect
largeness in finitely presented groups. There are partial algorithms
which will terminate with the answer yes if a finitely presented group
is large and which may return no or not terminate otherwise, see for
instance \cite{mear}. If we have a complete algorithm for largeness
then, as pointed
out by D.\,Groves and I.\,Leary, we have yes to this question because a group
$G$ has proper finite index subgroups if and only if $G*G*G$ is large.\\
\hfill\\
{\bf 13.} If a finitely presented group has infinite virtual first Betti
number then must it be large? Does it contain $F_2$? Must it be not
virtually soluble?\\
{\bf Notes:} These seem interesting questions. The first point to make is
that these are all false for finitely generated groups, as in \cite{edph}
it is pointed out that the soluble and $\mathbb R$-linear group
$\z\wr\z$ has infinite virtual first Betti number. Therefore one would
expect to find counterexamples by taking a group surjecting to $\z\wr\z$.
But Baumslag shows in \cite{bmbk} Chapter IV Theorem 7 
that a finitely presented group
surjecting to $\z\wr\z$ is large.\\
\hfill\\
We can also ask about growth of first Betti numbers in finite covers.\\
{\bf 14.} If a finitely presented group $G$ has a sequence of
finite index subgroups $H_n$ such that $\beta_1(H_n)/[G:H_n]\geq c>0$
for all $n$ then is $G$ large?\\
{\bf Notes:} If $H_n\unlhd G$ with the sequence nested and such that
$\cap H_n$ is trivial (thus implying that $G$ is residually finite)
then this limit exists and is the first $L^2$-Betti number (at least
if $G=\pi_1(X)$ for $X$ a CW-complex of finite type). However
groups of the form $F_n\rtimes_\alpha\z$ have zero first $L^2$-Betti number
but any large example will have a sequence as in the question.

If the sequence $(H_n)$ above is such that there is a surjective
homomorphism $\chi:G\rightarrow\z$ with $K=\mbox{ker }\chi\leq H_n$
for all $n$ then $G$ is large. Indeed we only require that 
$\beta_1(H_n)$ is unbounded, or even $\beta_1(H_n;\z/p\z)$ for some
prime $p$. This is because
\[\beta_1(H_n;\mathbb F)\leq\beta_1(K;\mathbb F)+1\mbox{ for }\mathbb F=\q
\mbox{ or }\z/p\z\]
so if we have an unbounded sequence then $\Delta_{G,\chi}^\mathbb F$ is
zero, with Howie's result implying that for all large $k$ the subgroup
$KG^k$ has a surjection onto $C_p*C_p*C_p$ (or $\z*\z*\z$ for 
$\mathbb F=\q$).
However the counterexample $\z\wr\z$ above for finitely generated groups
also has the first Betti number growing linearly in a sequence of
subgroups containing ker $\chi$.

One could also formulate variations of question 14 using first
Betti numbers modulo a prime $p$. However question 13 is no longer
true in this more general setting: consider the $\mathbb R$-linear
Baumslag-Remeslennikov soluble group
\[G=\langle a,s,t|tat^{-1}=asas^{-1},[a,sas^{-1}]=1=[s,t]\rangle.\]
By using the binomial theorem modulo any prime $p$, it is not hard to
show that we have subgroups $H_n$ with $\beta_1(H_n;\z/p\z)\geq 2+2p^{n-1}$
where $G'\unlhd H_n\unlhd G$ is the finite abelian cover corresponding to
\begin{eqnarray*}
s^{-p^{n-1}}(t^{p^n+p^{n-1}})
               \mbox{ and }s^{3p^{n-1}}&\mbox{ for }&p\equiv 1\mbox{ mod }3,\\
t^{p^n+p^{n-1}}\mbox{ and }s^{3p^{n-1}}&\mbox{ for }&p\equiv 2\mbox
{ mod }3,\\  
t^{p^n+p^{n-1}}\mbox{ and }s^{8p^{n-1}}&\mbox{ for }&p=3.
\end{eqnarray*}
This is the fastest possible growth of $\beta_1(H;\z/p\z)$ because Corollary
1.4 in \cite{w252} states that if $G$ is finitely presented and soluble
then there exists $\kappa >0$ such that $d(H/H')\leq\kappa |G:H|^{1/2}$ for
all $H\leq_f G$.
\hfill\\
{\bf 15.} Fix an integer $k\leq 1$. Is there $f(k)$ such that whenever a
group $G$ has a presentation of deficiency $k$ and $\beta_1(G)>f(k)$
then $G$ is large?

This seems less and less likely as $k$ gets smaller. It is of course
true for $k\geq 2$.\\
\hfill\\
{\bf 16.} (J.\,S.\,Wilson \cite{w252}) If a group $G$ has a finite
presentation $P$ such that def$(P)+d^2/4-d>0$ where $d\geq 2$ is the
minimum number of generators of $G/G'$ then does $G$ contain a
non-abelian free group?\\
{\bf Notes:} Bartholdi has recently proved in \cite{bth} that
Golod-Shafarevich groups are amenable. These are groups with a presentation 
(of finitely many generators but possibly infinitely many relators)
possessing a certain condition which is certainly satisfied by finite
presentations of the form above. This can be seen as providing evidence
for the answer yes because a counterexample would have the extremely
rare property of being a finitely presented
non-amenable group without $F_2$ as a subgroup. Lackenby and Lubotzky
ask for a group with a finite presentation as above which is not large:
surely these exist? There is also the question (credited to Lubotzky
and Zelmanov) of whether a group with such a presentation can have
property $(T)$ or $(\tau)$. Recently Ershov in \cite{ers} has constructed
Golod-Shafarevich groups with property $(T)$.\\
\hfill\\
{\bf Deficiency 1 Groups}\\
{\bf 17.} Is every deficiency 1 group that does not contain $F_2$
isomorphic to $BS(1,m)$?\\
{\bf Notes:} For evidence, see Section 3.\\
\hfill\\
{\bf 18.} Is every residually finite deficiency 1 group that is not large
isomorphic to $BS(1,m)$?\\
{\bf Notes:} In particular are all word hyperbolic groups of the form
$F_n\rtimes_\alpha\z$ for $n\geq 2$ large? A proof could be attempted in two
parts: assume the virtual
first Betti number is at least 2 to obtain largeness,
and separately establish this assumption which is question 12.16
in \cite{bspl} due to Casson.\\
\hfill\\
{\bf 19.} If a deficiency 1 group $G$ has (virtual)
first Betti number at least 3
then is it large? If the minimum number of generators of $G/G'$ 
(or $H/H'$ for $H\leq_f G$) is at least 3 then is $G$ large?\\
{\bf Notes:} The first part 
is question 15 for $k=1$ and $f(1)=2$ which is certainly more believable
here. Corollary 4.6 shows that we cannot replace ``(virtual)
first Betti number at
least 3'' with ``(virtual)
first Betti number at least 2 and $G\neq \z\times\z$ (or
$BS(1,-1)$)'', even in the
1-relator case.\\
\hfill\\
{\bf 2-generator 1-relator presentations}\\
In this special case question 17 is true but questions 18 and 19 are
unknown.\\
\hfill\\
{\bf 20.} Let $G$ be a group with a 2-generator 1-relator presentation
where the relator is not a proper power. Suppose that $G$ is residually
finite then does $G$
have a finite index subgroup $H$ which is an ascending HNN extension
of a finitely generated free group?\\
{\bf Notes:}  If so then every finite index subgroup of $H$, and $G$, has
deficiency 1. However if the relator is a proper power, which is exactly
when $G$ has torsion, then often $G$ has a finite index subgroup of
deficiency at least 2 and in any case $G$ is known to be large
(although \cite{bamo} Problem 4 asks if a 1-relator group with torsion
is virtually free-by-cyclic: here the free part could be infinitely
generated and being cyclic includes trivial here). But if $G$ is torsion
free then it has geometric dimension 2 so all finite index subgroups
have deficiency 1.

If we have a yes answer to this question then by \cite{fh} $H$, and so $G$, is
coherent which is a special case of the old question of whether all 1-relator
groups are coherent.\\
\hfill\\
{\bf 21.} (Borisov and Sapir \cite{bsp})
Is the property of being residually finite generic amongst
2-generator 1-relator presentations?\\
{\bf Notes:} It is shown in \cite{dt} 
Theorem 6.1 that a generic presentation is
not an ascending HNN extension of a finitely generated free group, but
experimentally 94\% of presentations give groups of the form $F_n\rtimes\z$.
There is an algorithm, due to Moldavanski\u{i} 
\cite{mol} in the case $\beta_1(G)=1$ and Brown \cite{bks} for $\beta_1(G)=2$,
as to whether a 2-generator 1-relator group is of this form. One could try
looking for $H\leq_f G$ with $H$ an ascending HNN extension of a finitely
generated free group (which is shown to be residually finite in \cite{bsp})
but the problem is of course that although $H$ will have a deficiency 1
presentation, it will not in general be 2-generated (indeed as soon as
$d(H/H')\geq 3$ it cannot be).

We can also ask about genericity for other properties: being linear,
being coherent (as evidence for the general conjecture) and being large.
It is known that a generic presentation gives rise to a group that is
one ended, torsion free and word hyperbolic, so if $G=F_n\rtimes\z$ is
large for word hyperbolic $G$ with $n\geq 2$ as in question 18 then we
would expect at least 94\% of random presentations to give large groups.\\
\hfill\\
{\bf 22.} Given reduced words $u,v\in F_2=\langle x,y\rangle$ which are
not a generating pair, does there exist a homomorphism $\rho$ from $F_2$
into the symmetric group $\Sigma_n$ for some $n$ such that $\rho(u)$
and $\rho(v)$ commute but $\rho(x)$ and $\rho(y)$ do not?\\
{\bf Notes:} This is equivalent to asking: if the group $G=\langle x,y|
[u,v]\rangle$ is proabelian then is it abelian (and thus equal to
$\z\times\z$)? Again there are interesting consequences because if so
then $G$ is either large or $\z\times\z$ by \cite{me} Theorem 3.6
but if not then
\cite{nypl} Question (OR8) is answered, which asks if $G$ is always residually
finite.\\
\hfill\\
{\bf 23.} (M\"uller and Schlage-Puchta \cite{mu} Problem 5) 
Is there an algorithm to detect
whether a 2-generator 1-relator presentation is large?

This special case of question 12 might be more likely to yield the
answer yes. However even if there is a straightforward criterion for
largeness of such presentations, it might not convert into an algorithm:
for instance if question 19 is true then it will detect largeness
(which we can do anyway for an arbitrary finitely presented group) but
would fail to prove that a given presentation is not large. We believe that
Section 4 provides evidence that there is not an obvious algorithm.


\begin{thebibliography}{99}

\bibitem{bth} L.\,Bartholdi,
On amenability of group algebras, II: graded algebras.
Available at
\texttt{http://arxiv.org/abs/math/0611709}

\bibitem{bam} G.\,Baumslag,
A non-cyclic one-relator group all of whose finite quotients are cyclic.
{\it J. Austral. Math. Soc.}. {\bf 10} (1969), 497--498.

\bibitem{bmbk} G.\,Baumslag,
{\it Topics in combinatorial group theory},
Lectures in Mathematics ETH Z\"urich. Birkh\"auser Verlag, Basel 1993.

\bibitem{bamo} G.\,Baumslag,
Some open problems. In {\it
Summer School in Group Theory in Banff, 1996}, CRM Proc. Lecture Notes,
17, Amer. Math. Soc., Providence 1999, 1--9.

\bibitem{bp} B.\,Baumslag and S.\,J.\,Pride,
Groups with two more generators than relators.
{\it J. London Math. Soc.}. {\bf 17} (1978), 425--426.

\bibitem{bspl} M.\,Bestvina,
Questions in Geometric Group Theory.
Available at
\texttt{http://www.math.utah.edu/\textasciitilde bestvina}

\bibitem{bf} M.\,Bestvina and M.\,Feighn,
A combination theorem for negatively curved groups.
{\it J. Differential Geom.}. {\bf 35} (1992), 85--101.

\bibitem{bfad} M.\,Bestvina and M.\,Feighn,
Addendum and correction to
``A combination theorem for negatively curved groups''.
{\it J. Differential Geom.}. {\bf 43} (1996), 783--788.

\bibitem{bh} M.\,Bestvina and M.\,Handel,
Train tracks and automorphisms of free groups.
{\it Ann. of Math.}. {\bf 135} (1992), 1--51.

\bibitem{binw} R.\,Bieri,
Deficiency and the geometric invariants of a group.
{\it J. Pure Appl. Algebra}. {\bf 208} (2007), 951--959.

\bibitem{bns} R.\,Bieri, W.\,D.\,Neumann and R.\,Strebel,
A geometric invariant of discrete groups.
{\it Invent. Math.}. {\bf 90} (1987), 451--477.

\bibitem{bi} R.\,Bieri and R.\,Strebel,
Almost finitely presented soluble groups.
{\it Comment. Math. Helv.}. {\bf 53} (1978), 258--278.

\bibitem{bsp} A.\,Borisov and M.\,Sapir,
Polynomial maps over finite fields and residual finiteness of mapping
tori of group endomorphisms.
{\it Invent. Math.}. {\bf 160} (2005), 341--356.

\bibitem{bormold} A.\,V.\,Borshchev and D.\,I.\,Moldavanski\u{i},
On the isomorphism of some groups with one defining relation.
{\it Math. Notes}. {\bf 79} (2006), 31--40.

\bibitem{brcp} N.\,Brady and J.\,Crisp,
CAT(0) and CAT(-1) dimensions of torsion free hyperbolic groups.
{\it Comment. Math. Helv.}. {\bf 82} (2007), 61--85.

\bibitem{brink} P.\,Brinkmann, 
Hyperbolic automorphisms of free groups.
{\it Geom. Funct. Anal.}. {\bf 10} (2000), 1071--1089.

\bibitem{bks} K.\,S.\,Brown,
Trees, valuations, and the Bieri-Neumann-Strebel invariant.
{\it Invent. Math.}. {\bf 90} (1987), 479--504.

\bibitem{brg} K.\,S.\,Brown and R.\,Geoghegan,
An infinite-dimensional torsion-free $FP_{\infty}$ group.
{\it Invent. Math.}. {\bf 77} (1984), 367--381.

\bibitem{bru} A.\,M.\,Brunner,
A group with an infinite number of Nielsen inequivalent one-relator
presentations.
{\it J. Algebra}. {\bf 42} (1976), 81--84.

\bibitem{brunclss} A.\,M.\,Brunner,
On a class of one-relator groups.
{\it Canad. J. Math.}. {\bf 32} (1980), 414--420.

\bibitem{butmp} J.\,O.\,Button,
Mapping tori with first Betti number at least two.
{\it J. Math. Soc. Japan}. {\bf 59} (2007), 351--370.

\bibitem{mear} J.\,O.\,Button,
Large mapping tori of free group endomorphisms.
Available at
\texttt{http://arxiv.org/abs/math/0511715}

\bibitem{mebs} J.\,O.\,Button,
A formula for the normal subgroup growth of Baumslag-Solitar
groups.
{\it J. Group Theory} (to appear).

\bibitem{me} J.\,O.\,Button,
Large groups of deficiency 1.
{\it Israel J. Math.} (to appear).

\bibitem{mek} J.\,O.\,Button,
Finite covers of the infinite cyclic cover of a knot.
{\it J. Knot Theory Ramifications} (to appear).

\bibitem{dun} N.\,M.\,Dunfield,
Alexander and Thurston norms of fibered 3-manifolds.
{\it Pacific J. Math.}. {\bf 200} (2001), 43--58.

\bibitem{dt} N.\,M.\,Dunfield and D.\,P.\,Thurston,
A random tunnel number one 3-manifold does not fiber over the
circle.
{\it Geom. Topol.}. {\bf 10} (2006), 2431--2499.

\bibitem{edph} M.\,Edjvet, 
The concept of ``Largeness'' in Group Theory.
Ph.D. thesis, University of Glasgow, Glasgow 1984.

\bibitem{ep} M.\,Edjvet and S.\,J.\,Pride,
The concept of ``largeness'' in group theory II. In
{\it Groups -- Korea 1983,} 
Lecture Notes in Math. {\bf 1098}, Springer, Berlin 1984, 29--54.

\bibitem{ers} M.\,Ershov,
Golod-Shafarevich groups with property $(T)$ and Kac-Moody groups.
Available at
\texttt{http://www.math.uchcago.edu/\textasciitilde ershov/Research/}

\bibitem{fh} M.\,Feighn and M.\,Handel,
Mapping tori of free group automorphisms are coherent.
{\it Ann. of Math.}. {\bf 149} (1999), 1061--1077.

\bibitem{gel} E.\,Gelman,
Subgroup growth of Baumslag-Solitar groups.
{\it J. Group Theory}. {\bf 8} (2005), 801--806.

\bibitem{geris181} S.\,M.\,Gersten,
Isoperimetric and isodiametric functions of finite presentations.
In {\it Geometric group theory, Vol. 1 (Sussex, 1991)}, London Math. Soc.
Lecture Note Ser. 181, Cambridge Univ. Press, Cambridge 1993, 79--96.

\bibitem{grst} S.\,M.\,Gersten and J.\,R.\,Stallings,
Irreducible outer automorphisms of a free group.
{\it Proc. Amer. Math. Soc.}. {\bf 111} (1991), 309--314.

\bibitem{hig} G.\,Higman,
A finitely generated infinite simple group.
{\it J. London Math. Soc.}. {\bf 26} (1951), 61--64.

\bibitem{hil} J.\,A.\,Hillman,
Tits alternatives and low dimensional topology.
{\it J. Math. Soc. Japan}. {\bf 55} (2003), 365--383.     

\bibitem{hilea} J.\,A.\,Hillman and P.\,A.\,Linnell,
Elementary amenable groups of finite Hirsch length are locally-finite
by virtually-soluble.
{\it J. Austral. Math. Soc. Ser. A}. {\bf 52} (1992), 237--241.

\bibitem{houg} C.\,H.\,Houghton,
The first cohomology of a group with permutation module
coefficients.
{\it Arch. Math. (Basel)}. {\bf 31} (1978/79), 254--258.

\bibitem{how} J.\,Howie,
Free subgroups in groups of small deficiency.
{\it J. Group Theory}. {\bf 1} (1998), 95--112.

\bibitem{jeaw} S.\,C.\,Jeanes and J.\,S.\,Wilson,
On finitely generated groups with many profinite-closed subgroups.
{\it Arch. Math. (Basel)}. {\bf 31} (1978/79), 120--122.

\bibitem{ko} D.\,H.\,Kochloukova,
Some Novikov rings that are von Neumann finite and knot-like groups.
{\it Comment. Math. Helv.}. {\bf 81} (2006), 931--943.

\bibitem{kou} 
{\it The Kourovka notebook. Unsolved problems in group theory},
Sixteenth edition, Edited by V.\,D.\,Mazurov and E.\,I.\,Khukhro.
Russian Academy of Sciences Siberian Division, Institute of Mathematics,
Novosibirsk 2006.

\bibitem{lr} D.\,D.\,Long and A.\,W.\,Reid,
Subgroup separability and virtual retractions of groups.
Available at
\texttt{http://www.math.utexas.edu/users/areid/publications.html}

\bibitem{luta} A.\,Lubotzky,
{\it Discrete groups, expanding graphs and invariant measures}.
Progress in Mathematics 125, Birkh\"auser Verlag, Basel 1994.

\bibitem{lub} A.\,Lubotzky,
Free Quotients and the first Betti number of some hyperbolic
manifolds. {\it Transform. Groups}. {\bf 1} (1996), 71--82.

\bibitem{lumn} A.\,Lubotzky and A.\,Mann,
Residually finite groups of finite rank.
{\it Math. Proc. Cambridge Philos. Soc.}. {\bf 106} (1989), 385--388.

\bibitem{ls} R.\,C.\,Lyndon and P.\,E.\,Schupp,
{\it Combinatorial Group Theory}. Springer-Verlag, Berlin-Heidelberg-New
York 1977.

\bibitem{mks} W.\,Magnus, A.\,Karrass and D.\,Solitar,
{\it Combinatorial Group Theory},
Dover Publications, Inc., Mineola, New York 2004.

\bibitem{mol} D.\,Moldavanski\u{i},
Certain subgroups of groups with one defining relation.
{\it Sibirsk. Mat. \u{Z}}. {\bf 8} (1967), 1370--1384.

\bibitem{mdsb} D.\,Moldavanski\u{i} and N.\,Sibyakova,
On the finite images of some one-relator groups.
{\it Proc. Amer. Math. Soc.}. {\bf 123} (1995), 2017--2020.

\bibitem{mu} T.\,W.\,M\"uller and J.-C.\,Schlage-Puchta,
Some examples in the theory of subgroup growth.
{\it Monatsh. Math.}. {\bf 146} (2005), 49--76.

\bibitem{nypl} New York Group Theory Cooperative,
Open problems in combinatorial and geometric group theory. Available at
\texttt{http://zebra.sci.ccny.cuny.edu/web/nygtc/problems/}

\bibitem{ol} A.\,Yu.\,Ol'shanski\u\i,
SQ-universality of hyperbolic groups.
{\it Sb. Math.}. {\bf 186} (1995), 1199--1211.

\bibitem{olsa} A.\,Yu.\,Ol'shanski\u{i} and M.\,V.\,Sapir,
Non-amenable finitely presented torsion-by-cyclic groups.
{\it Publ. Math. Inst. Hautes \'Etudes Sci.}. {\bf 96} (2003), 43--169.

\bibitem{rob} D.\,J.\,S.\,Robinson,
{\it A course in the theory of groups}, Second edition,
Graduate Texts in Mathematics 80, Springer-Verlag, New York 1996.

\bibitem{st} J.\,R.\,Stallings,
Topologically unrealizable automorphisms of free groups.
{\it Proc. Amer. Math. Soc.}. {\bf 84} (1982), 21--24.

\bibitem{w252} J.\,S.\,Wilson,
Finitely presented soluble groups. In
{\it Geometry and cohomology in group theory (Durham, 1994)},
London Math. Soc. Lecture Note Ser.  252, Cambridge Univ. Press, 
Cambridge 1998, 296--316.
\bibitem{zel} E.\,Zelmanov,
On groups satisfying the Golod-Shafarevich condition. In
{\it New horizons in pro-$p$ groups}, Progr. Math. 184, Birkh\"auser
Boston, Boston 2000, 223--232.

\end{thebibliography}
\end{document}